\documentclass[reqno,12pt]{amsart}
\usepackage[all,line,poly,rotate,ps,dvips]{xy}
\SelectTips{cm}{}

\usepackage{amssymb,eucal,graphicx}
\usepackage{enumerate}

\numberwithin{equation}{section}

\newtheorem{theorem}[equation]{Theorem}

\newtheorem{corollary}[equation]{Corollary}
\newtheorem{claim}[equation]{Claim}
\newtheorem{lemma}[equation]{Lemma}
\newtheorem{proposition}[equation]{Proposition}

\newtheorem{ass}[equation]{Assumptions}

\theoremstyle{definition}

\newtheorem{definition}[equation]{Definition}
\newtheorem{remark}[equation]{Remark}

\newtheorem{example}[equation]{Example}

\theoremstyle{remark}

\newcommand{\C}{\mathbb{C}}
\newcommand{\ZZ}{\mathbb{Z}}
\newcommand{\G}{\mathbb{G}}

\newcommand{\Pp}{\mathbb{P}}

\newcommand{\Ff}{\mathcal{F}}
\newcommand{\Ef}{\mathcal{E}}

\newcommand{\HH}{\mathcal{H}}

\newcommand{\N}{\mathcal{N}}
\newcommand{\M}{\mathcal{M}}
\newcommand{\T}{\mathcal{T}}
\newcommand{\Oc}{\mathcal{O}}

\newcommand{\g}{\mathfrak{g}}

\begin{document}
\title[Special scrolls with general moduli]%
{Special scrolls whose base curve has general moduli}%
\author{A. Calabri, C. Ciliberto, F. Flamini, R. Miranda}

\email{calabri@dmsa.unipd.it}
\curraddr{DMMMSA, Universit\`a degli Studi di Padova\\
Via Belzoni, 7 - 35131 Padova \\Italy}

\email{cilibert@mat.uniroma2.it} \curraddr{Dipartimento di
Matematica, Universit\`a degli Studi di Roma Tor Vergata\\ Via
della Ricerca Scientifica - 00133 Roma \\Italy}

\email{flamini@mat.uniroma2.it} \curraddr{Dipartimento di
Matematica, Universit\`a degli Studi di Roma Tor Vergata\\ Via
della Ricerca Scientifica - 00133 Roma \\Italy}

\email{Rick.Miranda@ColoState.edu} \curraddr{Department of
Mathematics, 101 Weber Building, Colorado State University
\\ Fort Collins, CO 80523-1874 \\USA}

\thanks{
\\
The first three authors are members of G.N.S.A.G.A.\ at
I.N.d.A.M.\ ``Francesco Severi''.}

\begin{abstract} In this paper we study the Hilbert scheme of smooth, linearly normal, special scrolls under suitable
assumptions on degree, genus and speciality.
\end{abstract}

\maketitle


\section{Introduction}\label{S:Intro}

The classification of scroll surfaces in a projective space is a classical subject. Leaving aside
the prehistory, the classical reference is C. Segre \cite{Seg}. In this paper, Segre
sets the foundations of the theory and proves a number of basic results concerning
the family of unisecant curves of sufficiently general scrolls. Though important, Segre's results
are far from  being exhaustive and even satisfactory. This was already clear to F. Severi who, in \cite{Sev},
makes some criticism to some points of Segre's treatment,  poses
a good number of interesting questions concerning the classification of
families of scrolls  in a projective space, and proves some partial results on this subject.  In particular,
in the last two or three sections of his paper, Severi focuses on the study of what today we call the
{\em Hilbert scheme} of scrolls and the related map to the moduli space of curves.
Severi's paper does not seem to have received too much attention in modern times. By contrast
Segre's approach has been reconsidered and improved in recent times by various authors (for a non comprehensive
list, see the references).

The present paper has to be seen as a continuation of \cite{CCFMLincei}, \cite{CCFMnonsp} and
\cite{CCFMBN}. In \cite{CCFMLincei}, inspired by some papers of G. Zappa (cf. \cite{Zap1}, \cite{Zapp}) in turn motivated
by Severi, we studied the Hilbert scheme of linearly normal,
non-special scrolls and some of their degenerations. Recall
that a scroll $S$ is said to be {\em special} if  $h^1(\Oc_S(1)) >0$ and {\em non-special}
otherwise (cf. Definition \ref{def:spec}); the number $h^1(\Oc_S(1)) $ is called the {\em speciality} of the
scroll. The general linearly normal,
non-special scroll corresponds to a stable, rank-two
vector bundle on a curve with general moduli
(see \cite{CCFMnonsp}, cf.\ also \cite{APS}).  This and the degeneration techniques in \cite{CCFMLincei} enabled us to
reconstruct and improve some enumerative results of Segre, reconsidered also by Ghione (cf. \cite{Ghio}).
The same techniques turn out to be very effective in studying the Brill-Noether theory of sub-line bundles of
a stable, rank-two vector bundle on a curve with general moduli, which translates into the study of families of unisecants
of the corresponding scroll (cf. \cite{CCFMBN}).

In the present paper, we focus on the study of the Hilbert scheme of linearly normal,
special scrolls. In contrast with what happens
for non-special scrolls, the special ones, in the ranges considered by us for the degree and genus,
always correspond to unstable --- even sometimes decomposable --- vector bundles (cf. Proposition \ref{prop:cases} and Remark \ref{rem:conclusion}, cf. also \cite{GP2}).

Our main result is Theorem \ref{thm:hilschh1}, in which we give, under suitable assumptions
on degree $d$, genus $g$ and speciality $h^1$, the full description of all
components of the Hilbert scheme of smooth, linearly normal, special scrolls whose base curve
has general moduli. With a slight abuse of terminology, we will call such scrolls,
{\em scrolls with general moduli}. Similarly, we will talk about {\em scrolls with special moduli}.

All components of the Hilbert scheme of special scrolls with general moduli
turn out to be generically smooth and we compute their dimensions. If the speciality
is 1, there is a unique such component. Otherwise, there are several components, each one determined by
the minimum degree $m$ of a section contained by the general surface of the component. If the speciality is 2,
all the components have the same dimension; if the speciality is larger than 2,
the components have different dimensions depending on $m$: the larger $m$ is the smaller is the dimension of the component.
Sections \ref{ProjDege}, \ref{SSGM} and \ref{S:HSS} are devoted to the proof of this theorem
under suitable assumptions on $d $.

These results rely on the existence of a unique special section, which is the section of minimal
degree and of the same speciality of the scroll in the range of interest for us.
If $d \geq 4g-2$, this is one of the main results by Segre in \cite{Seg}, which relies on a smart
projective geometric argument, and could be easily
exposed in modern terminology. However, we will use the more recent results in \cite{GP3} and \cite{GP4},
which ensure the existence and uniqueness of the special section in a wider range than Segre's.

Section \ref{S:considerations} contains further remarks on the Hilbert schemes. First we prove
that in most cases projections of linearly normal, special scrolls to lower dimensional projective spaces
fill up, rather unexpectedly, irreducible components of the Hilbert schemes (see
Proposition \ref{prop:ciropasqua}; see also \cite{Ballico2}). Moreover, we construct irreducible components of the Hilbert scheme
parametrizing scrolls with  special moduli (cf. Example \ref{ex:specmod}). Our construction
provides all components of the Hilbert scheme of scrolls of speciality 2, regardless to their
number of moduli (cf. Proposition \ref{prop:h1=2}). Finally, in
Propositions \ref{prop:16apr} and \ref{prop:16aprb}, we provide examples of smooth, linearly
normal, special scrolls which correspond to singular points of their Hilbert schemes
(cf. also Remark \ref{rem:rinvio}).

There are, in our opinion, three main subjects of interest in this field which have not been treated in this paper:
(i) degenerations, (ii) Brill-Noether theory, (iii) special scrolls of {\em very low} degree, in particular special scrolls
corresponding to stable vector bundles. We believe that the first two subjects can be attacked with techniques
similar to the ones in \cite{CCFMLincei}, \cite{CCFMnonsp} and \cite{CCFMBN}. We hope to come back to this in the future.
The last question is very intriguing and certainly requires an approach different
from the one in this paper. Related results can be found in \cite{TB} (cf. also the expository paper
\cite{Tex}).

\section{Notation and preliminaries}\label{S:NotPre}

Let $C$ be a smooth, irreducible, projective curve
of genus $g$ and let $F \stackrel{\rho}{\to} C $ be a {\em
geometrically ruled surface} on $C$, i.e.\ $F = \Pp(\Ff)$, where $\Ff$ is a
rank-two vector bundle, equivalently a locally free sheaf, on $C$. We will set
$d:=  \deg(\Ff) = \deg(\det(\Ff))$.

As in \cite{CCFMnonsp}, we shall make the following:

\begin{ass}\label{ass:1} With notation as above,
\begin{itemize}
\item [(1)] $h^0(C, \Ff) = R+1$, with $R \geq 3$;
\item[(2)] the complete linear system $|\Oc_F(1)|$ is base-point-free and the morphism
$\Phi: F \to \Pp^R$ induced by $|\Oc_F(1)|$ is birational to its image.
\end{itemize}
\end{ass}

\begin{definition}\label{def:scroll} The surface
$$\Phi(F) :=S \subset \Pp^R$$is said to be a {\em scroll of degree
$d$ and of (sectional) genus $g$} on $C$ and $F$ is its minimal desingularization; $S$ is called the {\em
scroll determined by the pair} $(\Ff, C)$. Note that $S$ is smooth if and only if $\Ff$ is very ample.

For any $x \in C$, let $f_x := \rho^{-1}(x) \cong \Pp^1$;
the line $l_x := \Phi(f_x)$ is called a {\em ruling} of
$S$. Abusing terminology, the
family $\{ l_x \}_{ x \in C}$ is also called the {\em ruling} of $S$.
\end{definition}

For further details on ruled surfaces and on their projective geometry, we refer the reader to
\cite{GH}, \cite[\S\,V]{Ha}, and the references in \cite{CCFMnonsp}.

Let $F \stackrel{\rho}{\to} C$ be as above. Then, there is a section $i : C \hookrightarrow F$,
whose image we denote by $H$, such that $\Oc_F(H) = \Oc_F(1)$. Then
$${\rm Pic}(F) \cong \ZZ[\Oc_F(H)] \oplus \rho^*({\rm Pic}(C)).$$Moreover,
$${\rm Num} (F) \cong \ZZ \oplus \ZZ,$$generated by the classes of $H$ and $f$, satisfying $Hf=1$, $f^2 = 0$
(cf.\ \cite[\S\,5, Prop.\ 2.3]{Ha}). If
$\underline{d}\in {\rm Div}(C)$,  we will denote by $\underline{d} f$ the divisor $\rho^*(\underline{d})$
on $F$. A similar notation will be used when $\underline{d}\in {\rm Pic}(C)$.
Any element of ${\rm Pic}(F)$ corresponds to a divisor on $F$
of the form$$nH + \underline{d} f, \; n \in \ZZ,  \; \underline{d}\in {\rm Pic}(C).$$As an element of
${\rm Num}(F)$ this is
$$nH + d f, \; n \in \ZZ, \; d = {\rm deg}(\underline{d}) \in \ZZ.$$We will denote by $\sim$ the linear
equivalence and by $\equiv$ the numerical equivalence.

\begin{definition}\label{def:unisec} For any $\underline{d} \in {\rm Div}(C)$ such that $|H + \underline{d} f| \neq \emptyset$, any $B \in |H + \underline{d} f|$ is called a {\em unisecant curve} to
the fibration $F \stackrel{\rho}{\to} C$ (or simply of $F$).
Any irreducible unisecant curve $B$ of $F$ is smooth and is called a {\em section} of $F$.
\end{definition}

There is a one-to-one correspondence between sections $B$ of $F$ and surjections
$\Ff \to\!\!\!\to L$, with $L =L_B$ a line bundle on $C$ (cf.\ \cite[\S\;V, Prop.\ 2.6 and 2.9]{Ha}).
Then, one has an exact sequence
\begin{equation}\label{eq:Fund}
0 \to N \to \Ff \to L \to 0,
\end{equation}where $N$ is a line bundle on $C$. If $L =\Oc_C( \underline{m}) $, with
$\underline{m} \in {\rm Div}^m(C)$, then $m = H B$ and $B \sim H + ( \underline{m} - \det(\Ff)) f $.

For example, if $B \in |H|$, the associated exact sequence is
$$0 \to \Oc_C \to \Ff \to \det(\Ff) \to 0,$$
where the map $\Oc_C \hookrightarrow \Ff$ gives a global section
of $\Ff$ corresponding to the global section of $\Oc_F(1)$ vanishing on $B$.

With $B$ and $F$ as in Definition \ref{def:unisec}, from \eqref{eq:Fund}, one has
\begin{equation}\label{eq:Ciro410b}
\Oc_B(B) \cong N^{\vee} \otimes L
\end{equation} (cf.\ \cite[\S\,5]{Ha}). In particular,
$$B^2 = \deg(L) - \deg(N) = d - 2 \, \deg(N) = 2m - d.$$

One has a similar situation if $B_1$ is a (reducible) unisecant curve of $F$ such that $H B_1 = m$.
Indeed, there is a section $B \subset F$ and an effective divisor $\underline{a} \in {\rm Div}(C)$,
$ a:= \deg(\underline{a})$, such that$$B_1 = B + \underline{a} f,$$where $BH = m-a$. In particular there is a
line bundle $L = L_B$ on $C$, with $\deg(L) = m-a$, fitting in the exact sequence \eqref{eq:Fund}.
Thus, $B_1$ corresponds to the exact sequence
\begin{equation}\label{eq:Fund2}
0 \to N \otimes \Oc_C( - \underline{a}) \to \Ff \to L \oplus \Oc_{\underline{a}} \to 0
\end{equation}(for details, cf. \cite{CCFMnonsp}).

\begin{definition}\label{def:direct} Let $S$ be a scroll of degree $d$ and
genus $g$  corresponding to the  pair $(\Ff, C)$.
Let $B \subset F$ be a section and  $L $ as in \eqref{eq:Fund}.
Let $\Gamma := \Phi(B) \subset S$. If $\Phi|_B$ is birational to its image, then
$\Gamma$ is called  a {\em section} of the scroll $S$. If $\Phi|_B$ is finite of degree
$n$ to its image, then $\Gamma$ is called a $n$-{\em directrix} of $S$, and the general point
of $\Gamma$ has multiplicity at least $n$ for $S$ (cf.\ e.g.\ \cite[Def.\ 1.9]{GP2}). If
$\Phi|_B$ is not finite, then $S$ is a {\em cone}.

We will say that the pair
$(S,\Gamma)$ is {\em associated with} \eqref{eq:Fund} and that $\Gamma$
{\em corresponds to the line bundle} $L$ on $C$.
If $m = \deg(L)$, then $m = n h$ and $\Gamma$ has degree $h$; indeed, the map
$\Phi|_B$ is determined by the linear series $\Lambda \subseteq |L|$, which
is the image of the map$$H^0(\Ff) \to H^0(L).$$

Similar terminology can be introduced for reducible unisecant curves.

\end{definition}

When $g= 0$ we have rational scrolls and these are well-known (see
e.g.\ \cite{GH}). Thus, from now on, we shall  focus on the case $g
\geq 1$.

Since ruled surfaces and scrolls are the projective counterpart of
the theory of rank-two vector bundles, we finish this section by recalling
some basic terminology on vector bundles. For
details, we refer the reader to e.g.\ \cite{New} and
\cite{Ses}.

Let $\Ef$ be a vector bundle of rank $r \geq 1$ on $C$. The {\em slope} of $\Ef$,
denoted by $\mu(\Ef)$, is defined as
\begin{equation}\label{eq:slope}
\mu(\Ef) := \frac{\deg(\Ef)}{r}.
\end{equation}

A rank-two vector bundle $\Ff$ on $C$ is
said to be {\em indecomposable}, if it cannot be expressed as a
direct sum $L_1 \oplus L_2$, for some $L_i \in {\rm Pic}(C)$, $1 \leq i
\leq 2$, and {\em decomposable} otherwise. Furthermore, $\Ff$ is said to be:
\begin{itemize}
\item {\em semistable}, if for any sub-line bundle  $N \subset \Ff
$, $\deg(N) \leq \mu(\Ff)$;
\item {\em stable}, if for any sub-line bundle $N \subset \Ff
$, $\deg(N) < \mu(\Ff)$;
\item {\em strictly semistable}, if it is semistable and
there is a sub-line bundle $N \subset \Ff
$ such that $\deg(N) = \mu(\Ff)$;
\item {\em unstable}, if there is a sub-line bundle $N \subset \Ff
$ such that $\deg(N) > \mu(\Ff)$. In this case, $N$ is called a {\em destabilizing} sub-line bundle of $\Ff$.
\end{itemize}

\section{Preliminary results on scrolls}\label{S:PreRes}

In this section, we recall some preliminary results concerning scrolls of
degree $d$ and genus $g \geq 1$ (cf.\ \cite{Seg}, \cite{GP1}, \cite{GP2} and
\cite[\S\;3]{CCFMnonsp}). We will keep the notation introduced in \S\;\ref{S:NotPre}.

If $K_F$ denotes a canonical divisor of $F$, one has
$$K_F \sim - 2 H + (\omega_C \otimes {\rm det}(\Ff)) f,$$hence
$$K_F \equiv  -2 H + (d+2g-2)f.$$

From Serre duality and the Riemann-Roch theorem, we have:
$$R+1:= h^0(\Oc_F(1)) = d - 2 g + 2 + h^1(\Oc_F(1)).$$

\begin{definition}\label{def:spec}
The non-negative integer $h^1 (\Oc_F(1))$ is called the {\em speciality} of the scroll and will be denoted by
$h^1(F)$, or simply by $h^1$, if there is no danger of confusion.
Thus
\begin{equation}\label{eq:R}
R= d - 2 g + 1 + h^1,
\end{equation} and the pair $(\Ff, C)$ determines $S \subset \Pp^R$ as a
{\em linearly normal scroll of degree $d$, genus $g$ and
speciality $h^1$}. Such a scroll $S$ is said to be {\em special} if $h^1 > 0$,
{\em non-special} otherwise.
\end{definition}

This definition coincides with the classical one given by Segre in \cite[\S\;3, p.\ 128]{Seg}:
notice that Segre denotes by $n$ the degree
of the scroll, by $p$ the sectional genus and by $ i :=  g-h^1$. Obviously, a scroll $S$ determined
by a pair $(\Ff,C)$ is special if $\iota:= h^1(C, {\rm det}(\Ff)) >0$. We will call such scrolls
{\em strongly special scrolls}. In this paper, we will restrict our attention to special scrolls
of  degree strictly larger than $2g-2$, which are therefore not strongly special.

Since $R \geq 3$, then $d \geq 2g+2-h^1$. One has an upper-bound for the speciality (cf.\ \cite[\S\,14]{Seg},
\cite{Ghio}).

\begin{lemma}\label{lem:segre}
In the above setting, if $\det(\Ff)$ is non-special, then $h^1 \leq g$ and,
if $d \geq 2g +2$, the equality holds  if and only if $\Ff =
\Oc_C \oplus L$, in which case $\Phi = \Phi_{|\Oc_F(1)|} $ maps $F$ to a cone $S$
over a projectively normal curve of degree $d$ and genus $g$ in
$\Pp^{d-g}$.
\end{lemma}

\noindent
For a proof, the reader is referred to \cite[Lemma 3.5]{CCFMnonsp}.

\begin{remark}\label{rem:ss} \normalfont{It is not difficult to give similar bounds for the speciality
of strongly special scrolls. Since for such a scroll $3 \leq R = d-2g+1+h^1$ and $d \leq 2g-2$, one has
$h^1 \geq 4$.

If $\rho : H^0(C, \Ff) \to H^0(C, {\rm det}(\Ff))$ is the natural restriction map, then
$h^1 = \iota + g - {\rm dim}({\rm Coker} \,(\rho))$.

Using cones, one sees that every admissible value of $h^1$ can be obtained, in case
$\rho$ is surjective.
}
\end{remark}

From \eqref{eq:R} and from Lemma \ref{lem:segre}, we have
\begin{equation}\label{eq:boundh0}
d - 2g + 1 \leq  R \leq d -g + 1,
\end{equation}where the upper-bound is realized by cones if $d \geq 2g+2$,
whereas the lower-bound is attained by non-special scrolls (cf.\ \cite[Theorem 5.4]{CCFMLincei}).
Any intermediate value of $h^1$ can be realized, e.g.\ by
using decomposable vector bundles
(see, \cite[pp.\ 144-145]{Seg} and \cite[Example 3.7]{CCFMnonsp}).

\begin{definition}\label{def:lndirec} Let $\Gamma \subset S$ be a section
associated to \eqref{eq:Fund}. Then:
\begin{itemize}
\item[(i)] $\Gamma$ is said to be {\em special} if
$h^1(C, L)>0$;
\item[(ii)] $\Gamma$ is said to be {\em linearly normally embedded} if
$H^0(\Ff) \to\!\!\!\!\to H^0(L)$.
\end{itemize}
\end{definition}

\begin{proposition}\label{prop:oss3}(cf. \cite[Thm. 3.6 and Cor. 3.7]{GP1})
Let $S \subset \Pp^R $ be a linearly normal scroll of degree $d$,
genus $g \geq 1$ and speciality $h^1 >0$.
Then $S$ contains a special section of speciality $i \leq  h^1$.
\end{proposition}

We will need a more precise result, mainly contained in \cite{GP4}.
First we prove a lemma.

\begin{lemma}\label{lem:maxdeg} Let $C$ be a smooth, projective curve of genus $g \geq 2$,
with general moduli. Let $0 < h^1 < g$ be an integer.

If $|M|$ is a complete linear system on $C$ having
maximal degree and maximal dimension with respect to the condition $h^1(C,M) = h^1$,
then $$\dim |M| = \overline{h}\qquad \text{and}
\qquad \deg (M) = \overline{m}$$where
\begin{equation}\label{eq:maxdeg}
\overline{h}:= \lfloor \frac{g}{h^1} -1 \rfloor \;\,\; \text{and} \;\;\;
\overline{m}:= \lfloor \frac{g}{h^1} -1 \rfloor + g - h^1.
\end{equation}
\end{lemma}
\begin{proof} If we fix the index of speciality $i$ and the genus $g$,
the Brill-Noether number $\rho(g,r,d) = \rho(g, r, r+g-i) := g - (r+1) i$
is a decreasing function of $r$. Since $C$ has general moduli, there exists a complete, b.p.f. and
special $\g^h_m$, of given speciality $h^1 = i$, if and only if
\begin{equation}\label{eq:BN2}
\rho(g,h,m) = g - (h+1)h^1 \geq 0.
\end{equation}Hence,
the maximal dimension for such a linear series is given by  the integer $\overline{h}$
as in \eqref{eq:maxdeg}. Accordingly, we get the expression for the maximal degree.
\end{proof}

\begin{theorem}\label{thm:seg1e2} Let $C$ be a smooth, projective curve of genus $g \geq 3$.
Let $0 < h^1 < g$ and $d \geq 4g - 2h^1 - {\rm Cliff}(C) + 1$ be integers.
Let $S \subset \Pp^R$ be a smooth, linearly normal, special scroll of degree $d$ and
speciality $h^1$ on $C$. Then:

\noindent (i) $S$ contains a unique, special section $\Gamma$. Furthermore, $\Gamma$ is linearly normally
embedded, its speciality equals the speciality of $S$, i.e.
$h^1(\Gamma, \Oc_{\Gamma}(1)) = h^1$ and, if $deg( \Gamma) : = m$, then
\begin{equation}\label{eq:bound}
\Gamma \subset \Pp^h, \; {\rm with} \; 2h  \leq m = h + g - h^1.
\end{equation}Moreover $\Gamma$ is the curve on $S$, different from a ruling, with minimal degree.

If in addition $d \geq 4g-3$, then $\Gamma$ is the unique section
with non-positive self-intersection.

\noindent (ii) If $C$ has general moduli and $d \geq \frac{7g-\epsilon}{2} - 2h^1 + 2$, where
$ 0 \leq \epsilon \leq 1$ and $\epsilon \equiv g \mod{2}$, then
\begin{equation}\label{eq:BN1}
either \; \; g \geq 4  h^1 , \; h\geq 3, \; \; or \; \; g=3, \; h^1=1, \; h=2.
\end{equation}Moreover
\begin{equation}\label{eq:bound2}
2h  \leq (h+g-h^1) \frac{h}{h+1} \leq m = h + g - h^1.
\end{equation}Furthermore $\Gamma$ is the unique section
with non-positive self-intersection.

\end{theorem}

\begin{corollary}\label{cor:seg1e2} Let
$g \geq 4$, $0 < h^1 < g$ and $d \geq \frac{7g-\epsilon}{2} - 2h^1 + 2$ be integers, with $\epsilon$ as above.
Assume  that $g < 4 h^1$.
Then, if $\HH$ denotes any irreducible component of the Hilbert scheme of linearly normal,
smooth, special scrolls in $\Pp^{R}$ of degree $d$, genus
$g$ and  speciality $h^1$, the general point of $\HH$ parametrizes
a scroll with special moduli.
\end{corollary}

There are examples of such components of the Hilbert scheme (cf. Remark \ref{rem:19608}).

\begin{remark}\label{rem:seg1e2} If $m= 2h$ then, by Clifford's theorem,  either
$\Gamma$ is a canonical curve or it is hyperelliptic. However, the second possibility is not
compatible with smoothness. Therefore, only the first alternative holds if $S$ is smooth
(cf. \cite[p. 144]{Seg} and \cite[Example 3.7]{CCFMnonsp}).
\end{remark}

\begin{proof}[Proof of Theorem \ref{thm:seg1e2}] Part (i), except the assertion
about $\Gamma^2$, follows from \cite{GP3} and \cite{GP4}.
Note that \eqref{eq:bound} follows by Clifford's theorem.

The numerical hypotheses on $d$ in Part (ii) come from the expression of the
Clifford index for a curve with general moduli.
Since $S$ is smooth, then either $h \geq 3$ or $h=2$, $g =3$ and $h^1=1$,
which gives \eqref{eq:BN1}. Furthermore, \eqref{eq:BN2} reads
$m (h+1) \geq h (h+g+1)$. Then, \eqref{eq:bound2} holds since
$\frac{h}{h+1} (h+g+1) = h(1 + \frac{g}{h+1})$ and  $\frac{g}{h+1}\geq 1$.

From \eqref{eq:Ciro410b}, one has
$$\Gamma^2 =  2m-d$$and this is negative if either
$d \geq 4g-3$, for any curve $C$, or if $C$ has general moduli as in part (ii),
because of the assumptions on $d$ and
the bounds on $m$ in Lemma \ref{lem:maxdeg}.

Finally, let $\Gamma' \equiv \Gamma + \alpha l$ be a section different from $\Gamma$. Then
$0 \leq \Gamma'\cdot \Gamma = \Gamma^2 + \alpha$,
hence $\alpha \geq - \Gamma^2 > 0$. If
$0 \geq (\Gamma')^2 = \Gamma^2 + 2 \alpha = \Gamma'\cdot \Gamma + \alpha > \Gamma'\cdot\Gamma $
we get a contradiction. \end{proof}

\begin{remark}\label{rem:pedreira} It is not true that the minimal section $\Gamma$ in Theorem
\ref{thm:seg1e2} is the only special curve on $S$ (cf. e.g. \cite[p. 129]{Seg}). In \cite[Proposition 2.2]{GP1} the authors prove that all the irreducible bisecant sections are special with speciality $h^1$. It would be interesting to make an analysis of the speciality of multi-secant curves.
\end{remark}

Under the assumption of Theorem \ref{thm:seg1e2}, the existence of a
special section  provides
information on the rank-two vector bundle on $C$ determining the scroll $S$.

\begin{proposition}\label{prop:cases} Same hypotheses as in Theorem \ref{thm:seg1e2}, Part (i), with
$d \geq 4g-3$, or as in Part (ii), if $C$ has general moduli. Suppose
$S$ determined by a pair $(\Ff,C)$. Then $\Ff $ is unstable.
If, moreover, $d \geq 6g -5$ then $\Ff= L \oplus N$, where $\Ff \to\!\!\!\!\to L$ corresponds to the special
section as in Theorem \ref{thm:seg1e2}.
\end{proposition}
\begin{proof} Let $\Gamma$ be the section as in Theorem \ref{thm:seg1e2}, corresponding to
the exact sequence \eqref{eq:Fund}. One has $\mu(\Ff) = \frac{d}{2}$ and $\deg(N) = d-m$.
From $\Gamma^2 = 2m-d <0$, one has
${\rm deg}(N)  > \mu(\Ff)$, hence $\Ff$ is unstable.

If $d \geq 6g-5$, from \eqref{eq:Fund}, we have
$[\Ff] \in {\rm Ext}^1(L,N) \cong H^1(C, N \otimes L^{\vee})$. Since $L$ is special,
${\rm deg}(N \otimes L^{\vee}) = d - 2m \geq 2g-1$, thus
$N \otimes L^{\vee}$ is non-special, so \eqref{eq:Fund} splits.

\end{proof}

%
\section{Projections and degenerations}\label{ProjDege}
%

Let $S$ be a smooth, special scroll of genus $g$ and degree $d$ as in  Theorem \ref{thm:seg1e2},
determined by a pair $(\Ff,C)$. Let $\Gamma$ be the special section of $S$  as in
Theorem \ref{thm:seg1e2} and let \eqref{eq:Fund}
be the associated exact sequence.

Let $l = l_q$ be the ruling of $S$ corresponding to the point $q \in C$ and let $p$ be a point on $l$. We want to consider
the ruled surface $S' \subset \Pp^{R-1}$ which is the projection of $S$ from $p$. We will assume that
$S'$ is smooth.

\begin{proposition}\label{prop:ciro53a} In the above setting, one has:
\begin{itemize}
\item[(i)] if $p \; \;  |\!\!\!\!\! \in \Gamma$, then $S'$ corresponds to a pair
$(\Ff',C)$, where $\Ff'$ fits into an exact sequence
\begin{equation}\label{eq:Fund3}
0 \to N (-q) \to \Ff' \to L \to 0;
\end{equation}
\item[(ii)] when $\Ff$ varies in the set of all extensions as in \eqref{eq:Fund} and $p$ varies on $l$,
then $\Ff'$ varies in the set of all extensions as in \eqref{eq:Fund3};
\item[(iii)] if $p \in \Gamma$, then $S'$ corresponds to a pair
$(\Ff',C)$, where $\Ff'$ fits into an exact sequence
\begin{equation}\label{eq:Fund3b}
0 \to N \to \Ff' \to L (-q) \to 0;
\end{equation}
\item[(iv)] when $\Ff$ varies in the set of all extensions as in \eqref{eq:Fund}
then $\Ff'$ varies in the set of all extensions as in \eqref{eq:Fund3b}.

\end{itemize}

\noindent If $d \geq 4g - 2h^1 - {\rm Cliff}(C) + 2$, then the
exact sequence \eqref{eq:Fund3} (resp. \eqref{eq:Fund3b})
corresponds to the unique special section of $S'$.
\end{proposition}

\begin{proof} Assertions (i) and (iii) are clear since  $S'$ contains a section
$\Gamma' $ such that, in the former case $\Oc_{\Gamma'} (1) \cong L$, in the latter
$\Oc_{\Gamma'} (1) \cong L \otimes \Oc_C(-q)$. More specifically,
there are surjective maps
$${\rm Ext}^1(L, N(-q)) \cong H^1 (N (-q) \otimes L^{\vee} ) \longrightarrow
H^1 (N  \otimes L^{\vee} ) \cong {\rm Ext}^1(L, N)$$
$${\rm Ext}^1(L, N) \cong H^1 (N \otimes L^{\vee} ) \longrightarrow H^1 (N  (q) \otimes L^{\vee} ) \cong {\rm Ext}^1(L (-q), N).$$
The former map has the following geometric interpretation: any surface $S'$ corresponding to a pair
$(\Ff',C)$ as in \eqref{eq:Fund3} comes as a projection of a surface $S$
corresponding to a pair $(\Ff,C)$ as in \eqref{eq:Fund}. The surface $S$ is obtained by blowing-up
the intersection point of the ruling $l'$ corresponding to $q$ with the section $\Gamma'$ and then
contracting the proper transform of $l'$.
The latter map corresponds to the projection of the surface $S$ to $S'$ from the intersection point of $l$ with
$\Gamma$.

The surjectivity of the above maps imply assertions (ii) and (iv). The final assertion is clear.
\end{proof}

The following corollary will be useful later.

\begin{corollary}\label{cor:ciro53b} In the above setting, let $q$ be a point in $C$ such that
$L(-q)$ is very ample on $C$. Then there is a flat degeneration of $S$ to a smooth scroll $S'$ corresponding to
a pair $(\mathcal G, C)$ with $\mathcal G$ fitting in an exact sequence
$$0 \to N(q) \to \mathcal G \to L(-q) \to 0.$$
\end{corollary}
\begin{proof} Proposition \ref{prop:ciro53a} tells us that $S$ comes as a projection of a scroll
$T$ in $\Pp^{R+1}$ corresponding to a pair $(\Ff',C)$ with $\Ff'$ fitting in an exact sequence
$$0 \to N (q) \to \Ff' \to L \to 0.$$By moving the centre of projection along the ruling
of $T$ corresponding to $q$ and letting it go to the intersection with the
unisecant corresponding to $L$ and applying again Proposition \ref{prop:ciro53a},
we see that $S$ degenerates to a surface $S'$ as needed.
\end{proof}

\begin{remark}\label{rem:ciro35c} {\normalfont The hypotheses in Proposition \ref{prop:ciro53a} and Corollary
\ref{cor:ciro53b} are too strong. In fact, one only needs the existence of the exact sequence \eqref{eq:Fund}.
However, \eqref{eq:Fund3} (resp. \eqref{eq:Fund3b}) will not correspond in general to the minimal
section. Moreover, the assumption that $L(-q)$ is very ample in Corollary \ref{cor:ciro53b}
is not strictly necessary: in case this does not hold
the limit scroll $S'$ is simply not smooth. However, we will never need this later.
}
\end{remark}

%
\section{Special scrolls with general moduli}\label{SSGM}
%

From now on we will focus on smooth, linearly normal scrolls of degree
$d$, genus $g$ and speciality $ h^1 $, with general moduli.

In what follows we shall use the following result.

\begin{lemma}\label{lem:flam} Let $C$ be a smooth, projective curve of genus $g \geq 2$, with general moduli.
Let $L$ be a special line bundle on $C$ such that $|L|$ is a base-point-free $\g^r_m$, with $i:=
h^1(C, L) \geq 2$. Assume also that $[L] \in W^r_m(C)$ is general when the Brill-Noether number
\begin{equation}\label{eq:BNn}
\rho(g, L) := g - h^0(C,L) h^1(C,L)
\end{equation}is positive. Then the complete linear system $|\omega_C \otimes L^{\vee}|$ is base-point-free.
\end{lemma}

\begin{proof} Assume by contradiction that the base divisor of
$|\omega_C \otimes L^{\vee}|$ is $p_1, \ldots , p_h$, $ h \geq 1$. Then
$$|L + p_1 + \cdots + p_h|$$ is a $ \g^{r+h}_{m+h}$ on $C$; moreover the series
$|B|:= |K - L - p_1 - \cdots - p_h|$ is base-point-free. We have a rational map
$$\phi: W^{r}_m(C) \dasharrow W^{r+h}_{m+h}(C)$$which maps
the general $[L] \in W^{r}_m(C)$ to $[L + p_1 + \cdots + p_h] \in W^{r+h}_{m+h}(C)$. Let
$\Phi_L$ be the component of the fibre of $[L]$ via $\phi$ containing $[L]$
and let $[L']$ be a general point of $\Phi_L$. Hence
$L + p_1 + \cdots + p_h = L' + q_1 + \cdots + q_h$, for some $q_i \in C$, $1 \leq i \leq h$, which are
the base points of $|K - L'|$.  There is a rational map
$$\psi_L : \Phi_L \dasharrow {\rm Sym}^h(C)$$which maps $[L']$ to $q_1 + \cdots + q_h$. The map $\psi_L$
is clearly birational to its image. Therefore
$${\rm dim}(W^{r+h}_{m+h}(C)) \geq {\rm dim} ({\rm Im} (\phi)) \geq {\rm dim}(W^{r}_{m}(C)) - h =$$
$${\rm dim}(W^{r+h}_{m+h}(C)) + hi - h $$which is a contradiction to $i \geq 2$.
\end{proof}

The main result of this section is the following proposition.

\begin{proposition}\label{prop:tghilbh1}
Let $h^1$ and $g$ be positive integers as in \eqref{eq:BN1}. Let $d$ be as in Theorem \ref{thm:seg1e2} - (ii),
if $h^1 \neq 2$, whereas $d \geq 4g - 3$, if $h^1=2$.

Let $S \subset \Pp^R$ be a smooth, linearly normal, special scroll of degree $d$, genus
$g$, speciality $h^1$, with general moduli, and let $\Gamma$ be the unique, special section of $S$
as in Theorem \ref{thm:seg1e2}.

Let $m = {\rm deg}(\Gamma)$ and let $L$ be the line bundle on $C$ associated to $\Gamma$ as in \eqref{eq:Fund}.
Assume further that

\begin{itemize}
\item $L = \omega_C$, when $h^1 =1$,
\item $[L] \in W^h_m(C)$ is such that $|\omega_C \otimes L^{\vee}|$ is base-point-free (in particular
$[L] \in W^h_m(C)$ is general, when  $h^1 \geq 2$ and $\rho(g,L)$ as in \eqref{eq:BNn} is positive).
\end{itemize}

\noindent
If $\N_{S/\Pp^{R}}$ denotes the normal bundle of $S$ in $\Pp^{R}$, then:
\begin{itemize}
\item[(i)] $h^0( S, \N_{S/\Pp^{R}}) = 7(g-1) + (R +1) (R +1 - h^1)
+ (d-m-g+1)h^1 - (d-2m + g -1)$;
\item[(ii)] $h^1( S, \N_{S/\Pp^{R}}) = h^1(d-m-g+1) - (d-2m + g -1)$;
\item[(iii)] $h^2( S, \N_{S/\Pp^{R}}) = 0$.
\end{itemize}
\end{proposition}

\begin{proof}[Proof of Proposition \ref{prop:tghilbh1}] First notice that $4g-3 \geq
\frac{7g-\epsilon}{2} - 2 h^1 + 2$, so that Theorem \ref{thm:seg1e2}, Part (ii), can be applied also to the case
$h^1=2$.

First, we prove (iii).  Since $S$ is linearly normal, from the Euler sequence we get:
$$ \cdots \to H^0(\Oc_S(H))^{\vee} \otimes H^2(\Oc_S(H)) \to H^2 (\T_{\Pp^R}|_S) \to 0;$$since
$S$ is a scroll, then $h^2(\Oc_S(H)) = 0$, which implies $h^2 (\T_{\Pp^R}|_S)= 0$.
Thus (iii) follows from using the tangent sequence
\begin{equation}\label{eq:tang}
0 \to \T_S \to \T_{\Pp^{R}}|_S \to \N_{S/\Pp^{R}} \to 0.
\end{equation}

Since $S$ is a scroll,
\begin{equation}\label{eq:tgS}
\chi(\T_S) = 6 - 6g.
\end{equation}From the Euler sequence, since $S$ is linearly
normal, we get
\begin{equation}\label{eq:tgS2}
\chi(\T_{\Pp^{R}}|_S) = (R +1) (R +1 - h^1) + g-1.
\end{equation}Thus, from (iii) and
from \eqref{eq:tgS}, \eqref{eq:tgS2} we get
\begin{equation}\label{eq:tgS3bis}
\chi(\N_{S/\Pp^{R}}) = h^0(\N_{S/\Pp^{R}}) - h^1(\N_{S/\Pp^{R}})=
7(g-1) + (R +1) (R +1 - h^1).
\end{equation}

The rest of the proof will be concentrated on the computation of $h^1(\N_{S/\Pp^{R}})$.

\begin{claim}\label{cl:ciro1502} One has $h^1(\N_{S/\Pp^{R}} (-\Gamma)) = h^2(\N_{S/\Pp^{R}} (-\Gamma))  = 0$.
In other words,
\begin{equation}\label{eq:tgS15}
h^1(\N_{S/\Pp^{R}}) = h^1(\N_{S/\Pp^{R}}|_{\Gamma}).
\end{equation}
\end{claim}

\begin{proof}[Proof of Claim \ref{cl:ciro1502}] Look at the exact
sequence
\begin{equation}\label{eq:ciro2}
0 \to \N_{S/\Pp^{R}} (-\Gamma) \to  \N_{S/\Pp^{R}}
\to \N_{S/\Pp^{R}}|_{\Gamma} \to 0.
\end{equation} From \eqref{eq:tang} tensored by $\Oc_S(-\Gamma)$ we see that $h^2(\N_{S/\Pp^{R}} (-\Gamma)) = 0$
follows from $h^2( \T_{\Pp^R}|_S (-\Gamma)) = 0$ which, by the Euler sequence, follows from $h^2(\Oc_S(H - \Gamma)) =
h^0(\Oc_S(K_S - H + \Gamma)) = 0$,
since $K_S - H + \Gamma$ intersects the ruling of $S$ negatively.

As for $h^1(\N_{S/\Pp^{R}} (-\Gamma))= 0$, this follows from $h^1( \T_{\Pp^R}|_S (-\Gamma)) = h^2( \T_S (-\Gamma)) = 0$.
By the Euler sequence, the first vanishing follows from $h^2(\Oc_S(-\Gamma)) = h^1(\Oc_S(H-\Gamma))=0$.
Since $K_S+\Gamma$ meets the ruling negatively, one has $h^0(\Oc_S(K_S +\Gamma)) = h^2(\Oc_S(-\Gamma)) =0$. Moreover,
Theorem \ref{thm:seg1e2} implies $h^1(\Oc_S(H-\Gamma)) = h^1(C, N)=0$.

In order  to prove $h^2( \T_S (-\Gamma)) = 0$, consider the exact sequence
\begin{equation}\label{eq:tgrel}
0 \to \T_{rel} \to \T_S \to \rho^*(\T_C) \to 0
\end{equation}arising from the structure morphism
$S\cong F = \Pp(\Ff) \stackrel{\rho}{\to} C$. The vanishing we need follows from
$h^2( \T_{rel} \otimes \Oc_S(-\Gamma)) = h^2 (\Oc_S(-\Gamma) \otimes \rho^*(\T_C)) = 0$. The first vanishing
holds since $\T_{rel} \cong \Oc_S (2H - df)$ and therefore, $\Oc_S(K_S + \Gamma) \otimes \T_{rel}^{\vee}$ restricts
negatively to the ruling. Similar considerations also yield the second vanishing.

\end{proof}

Next, we want to compute $h^1( \Gamma, \N_{S/\Pp^{R}}|_{\Gamma} )$. To this aim,
consider the exact sequence
\begin{equation}\label{eq:tgS19}
0 \to \N_{\Gamma/S} \to  \N_{\Gamma/\Pp^{R}}\to \N_{S/\Pp^{R}}|_{\Gamma} \to 0.
\end{equation}First we compute $h^1(\N_{\Gamma/\Pp^{R}})$ and $h^1(\N_{\Gamma/S} )$.

\begin{claim}\label{cl:flam1} One has
\begin{equation}\label{eq:tgS00}
h^1(\N_{\Gamma/S}) = d - 2m + g - 1
\end{equation}and
\begin{equation}\label{eq:tgS26}
h^1( \N_{\Gamma/\Pp^R})  = (d - g  - m + 1) h^1.
\end{equation}

\end{claim}

\begin{proof}[Proof of Claim \ref{cl:flam1}] From \eqref{eq:Ciro410b},
$${\rm deg}(\N_{\Gamma/S}) =  \Gamma^2 = 2m -d <0$$by Theorem \ref{thm:seg1e2}.
Thus,$$h^0(\Oc_{\Gamma}(\Gamma)) = 0, \;\;\;\; h^1(\Oc_{\Gamma}(\Gamma)) = d - 2m + g - 1$$which
gives \eqref{eq:tgS00}.

To compute $h^1(\N_{\Gamma/\Pp^{R}})$ we use the fact that $S$ is a scroll with general moduli.
First, consider $\Gamma \subset \Pp^h$ and the Euler sequence of $\Pp^h$ restricted to $\Gamma$.
By taking cohomology and by dualizing, we get
$$0 \to H^1( \T_{\Pp^h}|_{\Gamma})^{\vee} \to H^0 (\Oc_{\Gamma}(H)) \otimes H^0(\omega_{\Gamma} (-H)) \stackrel{\mu_0}{\to}
H^0(\omega_{\Gamma}),$$where $\mu_0$ is the usual Brill-Noether map of $\Oc_{\Gamma}(H)$.
Since $\Gamma \cong C$ and since $C$ has general moduli, then
$\mu_0$ is injective by Gieseker-Petri (cf. \cite{ACGH})
so
\begin{equation}\label{eq:tgS22}
h^1( \T_{\Pp^h}|_{\Gamma})= 0.
\end{equation}From the exact sequence
$$0 \to \T_{\Gamma} \to  \T_{\Pp^h}|_{\Gamma}
\to \N_{\Gamma/\Pp^h} \to 0$$ we get
\begin{equation}\label{eq:tgS24}
h^1( \N_{\Gamma/\Pp^h})= 0.
\end{equation}

From the inclusions $$\Gamma \subset \Pp^h \subset \Pp^R$$we have the sequence
\begin{equation}\label{eq:tgS25}
0 \to \N_{\Gamma/\Pp^{h}} \to \N_{\Gamma/\Pp^{R}} \to \N_{\Pp^h/\Pp^{R}}|_{\Gamma} \cong \oplus_{i=1}^{R-h} \Oc_{\Gamma}
(H) \to 0.
\end{equation}By \eqref{eq:tgS24}, \eqref{eq:tgS25}, we have
\begin{displaymath}
\begin{array}{rcl}
h^1( \N_{\Gamma/\Pp^R}) & = & (R-h) h^1(\Oc_{\Gamma} (H)) = (d - 2g + 1 + h^1 - h) h^1 \\
 & = & (d - g  - m + 1) h^1,
\end{array}
\end{displaymath}proving \eqref{eq:tgS26}.

\end{proof}

Next we show that $$h^1(\N_{S/\Pp^{R}}|_{\Gamma}) = h^1(\N_{\Gamma/\Pp^{R}}) - h^1(\N_{\Gamma/S}).$$

\begin{claim}\label{cl:ciro1} The map
$$H^0(\Gamma, \N_{\Gamma/\Pp^{R}} )  \to  H^0( \Gamma, \N_{S/\Pp^{R}}|_{\Gamma} ),$$
coming from the exact sequence \eqref{eq:tgS19} is surjective.
\end{claim}

\begin{proof}[Proof of Claim \ref{cl:ciro1}] To show this surjectivity is equivalent to showing
the injectivity of the map
$$H^1 (\N_{\Gamma/S}) \to H^1(\N_{\Gamma/\Pp^{R}})$$or equivalently, from \eqref{eq:Ciro410b},
the surjectivity of the dual map
$$H^0(\omega_{\Gamma} \otimes \N_{\Gamma/\Pp^{R}}^{\vee}) \to H^0(\omega_{\Gamma} \otimes \N_{\Gamma/S}^{\vee})
\cong H^0(\omega_C \otimes N \otimes L^{\vee}).$$

From \eqref{eq:Fund} and the non-speciality of $N$, we get
$$0 \to H^0(L)^{\vee} \to H^0(\Ff)^{\vee} \to H^0(N)^{\vee} \to 0.$$Since
$H^0(\Oc_S(H)) \cong H^0(C, \Ff)$ and $\Oc_{\Gamma}(H) \cong L$, the Euler sequences
restricted to $\Gamma$ give the following commutative diagram:

\begin{displaymath}
\begin{array}{ccccccc}
      &             &     &     0                       &     &     0     &     \\
       &             &     &     \downarrow                        &     &     \downarrow      &     \\
0 \to & \Oc_{\Gamma} & \to & H^0(L)^{\vee} \otimes \Oc_{\Gamma}(H) & \to & \T_{\Pp^h}|_{\Gamma} & \to 0  \\
      &      ||       &     &     \downarrow                        &     &     \downarrow      &     \\
0 \to & \Oc_{\Gamma} & \to & H^0(\Ff)^{\vee} \otimes \Oc_{\Gamma}(H) & \to & \T_{\Pp^R}|_{\Gamma} & \to 0  \\
     &             &     &     \downarrow                        &     &     \downarrow      &     \\
 &  &  & H^0(N)^{\vee} \otimes \Oc_{\Gamma}(H) & \stackrel{\cong}{\to} & \N_{\Pp^h/\Pp^{R}}|_{\Gamma} & \\
          &             &     &     \downarrow                        &     &     \downarrow      &     \\
     &             &     &     0                       &     &     0     &

\end{array}
\end{displaymath}which shows that $\N_{\Pp^h/\Pp^{R}}|_{\Gamma} \cong H^0(N)^{\vee} \otimes \Oc_{\Gamma}(H)$ (cf. also \eqref{eq:tgS25}).

From \eqref{eq:tgS24}, \eqref{eq:tgS25} and the above identification, we have
$$H^1(\N_{\Gamma/\Pp^{R}}) \cong H^0(C, N)^{\vee} \otimes H^1(C, L)$$i.e.
$$H^0(C, N) \otimes H^0(C, \omega_C \otimes L^{\vee})\cong H^0(\omega_{\Gamma} \otimes \N_{\Gamma/\Pp^{R}}^{\vee}).$$By taking into account \eqref{eq:tgS19} and \eqref{eq:tgS25}, we get
a commutative diagram
\begin{equation}\label{eq:flammaggio}
\begin{array}{rccl}
H^0(N) \otimes H^0(\omega_C \otimes L^{\vee}) &  \stackrel{\cong}{\longrightarrow}& H^0(\omega_{\Gamma} \otimes \N_{\Gamma/\Pp^{R}}^{\vee}) &  \\
  & \searrow & \downarrow & \\
  & & H^0(\omega_C \otimes N \otimes L^{\vee}) & \cong H^0(\omega_{\Gamma} \otimes \N_{\Gamma/S}^{\vee}).
  \end{array}
\end{equation}Therefore, if the diagonal map is surjective then also the vertical one will be surjective.

There are two cases.

\noindent
$\bullet$ If $h^1=1$, then $L=\omega_C$ and the diagonal map is the identity; so we are done in this case.

\noindent
$\bullet$ If $h^1(L) \geq 3$, the surjectivity of the diagonal map
follows by applying \cite[Lemma 2.9]{Butler}. In the notation there,
$L_1 = \omega_C \otimes L^{\vee}$, $L_2 = N$, $p = 2g + 2 + 2m - d$ and
the conditions (ii) and (iii) in \cite[Lemma 2.9]{Butler}
ensuring the surjectivity follow from the hypothesis on $d$.

\noindent
$\bullet$ If $h^1(L) = 2$ the required surjectivity follows as above
by using Castelnuovo's Lemma (cf. \cite[Theorem 2]{Mumford} rather than \cite{Butler}. Indeed,
$|\omega_C \otimes L^{\vee}|$ is base-point-free and $h^1(N \otimes L\otimes \omega_C^{\vee}) = 0$
since $\deg( N \otimes L\otimes \omega_C^{\vee} ) = d - 2g + 2 \geq 2g-1$. In particular, for
$[L] \in W^r_m(C)$ general, one can conclude by using Lemma \ref{lem:flam}.

\end{proof}

From  Claims \ref{cl:ciro1502} and  \ref{cl:ciro1}, we get
$$h^1(\N_{S/\Pp^{R}}) = h^1(\N_{S/\Pp^{R}}|_{\Gamma} )  = h^1(\N_{\Gamma/\Pp^{R}}) - h^1(\N_{\Gamma/S} ).$$ Thus,
from Claim \ref{cl:flam1},  we get
$$h^1( \N_{S /\Pp^R})=  (d  - m - g + 1)h^1 - (d - 2m + g - 1)$$which is (ii) in the statement.
By using \eqref{eq:tgS3bis} and (ii), we get (i).
This completes the proof of Proposition \ref{prop:tghilbh1}.
\end{proof}

\begin{remark}\label{rem:tghilbh1} From \eqref{eq:tgS3bis}, we have:
$$\chi(\N_{S/\Pp^{R}}) = 7(g-1) + (R +1) (R +1 - h^1) = 7(g-1) + (R+1)^2 - h^1(L) h^0(\Ff).$$
Thus, formula (i) reads as
$$h^0(S, \N_{S/\Pp^{R}}) = \chi(\N_{S/\Pp^{R}}) + h^1(L) h^0(N) - \chi(N \otimes L^{\vee}) - 2 (g-1),$$or equivalently
\begin{equation}\label{eq:contoesplicito}
h^0(S, \N_{S/\Pp^{R}}) = 5(g-1) + (R+1)^2 - h^1(L) h^0(L) - \chi(N \otimes L^{\vee}).
\end{equation}
\end{remark}

\begin{remark}\label{rem:rinvio} The proof of Claim \ref{cl:ciro1} shows that if $|\omega_C \otimes L^{\vee}|$
has $t$ base-points then $$h^1(\N_{S/\Pp^{R}}) = h^1(\N_{\Gamma/\Pp^{R}}) - h^1(\N_{\Gamma/S}) + t$$and therefore
$$h^0(\N_{S/\Pp^{R}}) = 7(g-1) + (R +1) (R +1 - h^1)
+ (d-m-g+1)h^1 - (d-2m + g -1) + t.$$

\end{remark}

\begin{remark}  If one applies different surjectivity criteria for multiplication maps,
one may also have different ranges in which the previous proposition holds. For instance, by applying
\cite[Theorem 1]{Butler}, one may prove that the conclusion in Proposition \ref{prop:tghilbh1} holds
if $d > 4g-4h^1$ and $g > 6h^1 -2$.

\end{remark}

\begin{remark}\label{rem:conclusion} A very interesting question is to look for components of the Hilbert
scheme whose general point corresponds to a smooth, linearly normal, special scroll with general moduli
corresponding to a stable vector bundle. Such components are related to
irreducible components of Brill-Noether loci in $U_C(d)$, the moduli space
of degree $d$, semistable, rank-two vector bundles on $C$,
a curve of genus $g$ with general moduli.
There are several open questions on this subject: as a reference, see e.g. the overview in \cite{Tex}.
\end{remark}

\section{Hilbert schemes of linearly normal, special scrolls}\label{S:HSS}

From now on, we will denote by $ {\rm Hilb} (d,g,h^1)$ the open subset of the Hilbert
scheme parametrizing smooth
scrolls in $\Pp^R$ of genus $g \geq 3$, degree
$d \geq 2g+2$ and speciality $ h^1$, with $0 < h^1 <g$ and $R = d-2g + 1 + h^1$ as in \eqref{eq:R}.

Theorem \ref{thm:seg1e2} can be used to
describe the irreducible components of
${\rm Hilb}(d,g,h^1)$ whose general point represents a smooth,
scroll with general moduli. We will denote by $\HH_{d,g,h^1}$ the union of these components.

\begin{theorem}\label{thm:hilschh1}
Let $g \geq 3$, and $h^1>0$ be integers as in \eqref{eq:BN1}. Let
$d \geq \frac{7g-\epsilon}{2} - 2h^1 + 2$, where
$ 0 \leq \epsilon \leq 1$ and $\epsilon \equiv g \mod{2}$, if
$h^1 \neq 2$, whereas $d \geq 4g -3$, if $h^1 = 2$.
Let $m$ be any integer such that either
\begin{equation}\label{eq:nettu}
m= 4, \; {\rm if} \; g=3, h^1 = 1, \;\;\; {\rm or} \;\;\;
g + 3 - h^1 \leq m \leq \overline{m}:= \lfloor \frac{g}{h^1} -1 \rfloor + g - h^1, \;
{\rm otherwise}.
\end{equation}

\noindent
(i) If $h^1 =1$ then ${\rm Hilb}(d,g,1)$ consists of a unique, irreducible
component $\HH^{2g-2}_{d,g,1}$ whose general point parametrizes a smooth, linearly normal and special
scroll $S \subset \Pp^{R}$, $R = d-2g+2$, whose special section
is a canonical curve. Furthermore,
\begin{itemize}
\item[(1)] $dim(\HH^{2g-2}_{d,g,1}) = 7(g-1) + (d-2g+3)^2 - (d-2g+3)$,
\item[(2)] $\HH^{2g-2}_{d,g,1}$ is generically smooth and dominates ${\mathcal M}_g$.
\end{itemize}Moreover,
scrolls with speciality $1$, whose special section is not canonical of degree $m < 2g-2$
fill up an irreducible
subscheme of $\HH^{2g-2}_{d,g,1}$ which also dominates ${\mathcal M}_g$ and whose codimension
in $\HH^{2g-2}_{d,g,1}$ is $2g-2-m$.

\vskip 3pt

\noindent
(ii) If $ h^1 \geq 2$ then, for any $g \geq 4h^1$ and for any $m$ as in
\eqref{eq:nettu}, ${\rm Hilb}(d,g,h^1)$ contains
a unique, irreducible component $\HH_{d,g,h^1}^m$ whose general point parametrizes a smooth, linearly
normal and special scroll  $S \subset \Pp^{R}$, $R = d - 2g + 1 + h^1$,
having general moduli whose special section
$\Gamma$ has degree $m$ and speciality $h^1$. Furthermore,

\begin{itemize}
\item[(1)] $dim(\HH_{d,g,h^1}^m) = 7(g-1) + (R +1) (R +1 - h^1)
+ (d-m-g+1)h^1 - (d-2m + g -1),$
\item[(2)] $\HH_{d,g,h^1}^m$ is generically smooth.
\end{itemize}
\end{theorem}

\begin{remark}\label{rem:boundsm} Let us comment on the bounds on $m$ in \eqref{eq:nettu}.
From Theorem \ref{thm:seg1e2}, $S$ contains a unique,
special section $\Gamma$ of speciality $h^1$, which is the image of $C$ via a complete linear series
$|L|$, which is a $\g^h_m$. In order to have $C$ with general moduli, by Lemma \ref{lem:maxdeg}, $m \leq \overline{m} =
\lfloor \frac{g}{h^1} -1 \rfloor + g - h^1$. This condition is empty if $h^1=1$.

On the other hand, since $S$ is smooth, we have $h= m - g + h^1 \geq 2$. If $h^1 \geq 2$ and the scroll has general
moduli, then $h \geq 3$ by \eqref{eq:BN1}. If $h^1=1$ and $h=2$, then $m=4$ and $g=3$.

\end{remark}

\begin{remark}\label{rem:hilschh1} As a consequence of Theorem \ref{thm:hilschh1}, ${\rm Hilb}(d,g,h^1)$ is reducible
as soon as
$h^1 \geq 2$, and  when $h^1\geq 3$ it is  also not equidimensional. Indeed, the component of maximal dimension
 is $\HH_{d,g,h^1}^{g + 3 - h^1}$ whereas the component of minimal dimension is $\HH_{d,g,h^1}^{\overline{m}}$,
with $\overline{m}$ as in \eqref{eq:nettu}.

By Proposition \ref{prop:cases}, the scrolls  in $\HH_{d,g,h^1}^m$ correspond
to unstable bundles.

\end{remark}

\begin{proof}[Proof of Theorem \ref{thm:hilschh1}] The proof is in several steps.

\vskip 5pt

\noindent {\bf Step 1. Construction of $\HH_{d,g,h^1}^m$.}
%
As usual, for any smooth, projective curve $C$ of genus $g$ and for any integer $m$, denote by
$W^h_m(C)$ the subscheme of ${\rm Pic}^m(C)$
consisting of special, complete linear series $|A|$ on $C$ such that
${\rm deg}(A)=m$ and $h^0(C,A) \geq h+1$ and by $G^h_m(C)$ the scheme
parametrizing special $\g^h_m$'s on $C$ (see, e.g., \cite{ACGH}). Consider the morphism
$G^h_m(C) \to W^h_m(C)$.

If $C$ has general moduli, this map is birational and
$G^h_m (C)$ is smooth (see, e.g \cite{ACGH}). However, for any curve $C$, if $h= m - g +1$ and if
the general element of  a component $Z$ of $G^h_m (C)$ has index of speciality $1$, again the map
is birational and $Z$ is birational to ${\rm Sym}^{2g-2-m}(C)$, hence it is generically smooth. In any case,
for any  $m $ as in \eqref{eq:nettu},
$${\rm dim} (W^h_m (C)) = {\rm dim} (G^h_m (C)) = \rho(g,h,m)= g - (h+1) (h-m+g) \geq 0.$$

Let ${\mathcal M}^0_g$ be the Zariski open subset of the moduli space ${\mathcal M}_g$, whose points
correspond to equivalence classes of curves of genus $g$ without non-trivial automorphisms.
By definition, ${\mathcal M}^0_g$ is a fine moduli space, i.e.\ we have a universal family
$p : {\mathcal C} \to {\mathcal M}^0_g$, where ${\mathcal C}$ and ${\mathcal M}^0_g$ are smooth schemes and
$p$ is a smooth morphism. ${\mathcal C} $ can be identified with the Zariski open
subset ${\mathcal M}^0_{g,1}$ of the moduli space ${\mathcal M}_{g,1}$ of smooth, pointed, genus $g$
curves, whose points correspond to equivalence classes of pairs $(C,x)$, with $x \in C$ and $C$ a smooth
curve of genus $g$ without non-trivial automorphisms. On  ${\mathcal M}^0_{g,1}$ there is again a universal family
$p_1 : {\mathcal C}_1 \to {\mathcal M}^0_{g,1}$, where
${\mathcal C}_1 = {\mathcal C} \times_{{\mathcal M}^0_g} {\mathcal C}$. The family $p_1$ has a natural regular
global section $\delta$ whose image is the diagonal. By means of $\delta$, for any integer $k$, we have
the universal family of Picard varieties of order $k$,
i.e.\ $$p_1^{(k)} : {\mathcal Pic}^{(k)} \to {\mathcal M}^0_{g,1}$$
(cf.\ \cite[\S\,2]{Ciro}), with a Poincar\`e line-bundle on $
{\mathcal C}_1 \times_{{\mathcal M}^0_{g,1}}{\mathcal Pic}^{(k)}$ (cf. a relative version
of \cite[p. 166-167]{ACGH}).  For any closed point
$[(C,x)] \in {\mathcal M}^0_{g,1}$, its fibre via $p_1^{(k)}$ is isomorphic to ${\rm Pic}^{(k)}(C)$.

Then, one can consider the map
$${\mathcal W}^h_m \stackrel{p_1^{(m)}}{\longrightarrow} {\mathcal M}^0_{g,1},$$with
${\mathcal W}^h_m$ a subscheme of  ${\mathcal Pic}^{(m)}$: for any closed point
$[(C,x)] \in {\mathcal M}^0_{g,1}$, its fibre via
$p_1^{(m)} $ is isomorphic to $ W^h_m(C)$. One has:

\noindent
(a) if $\rho(g,h,m) >0$, by a result of Fulton-Lazarsfeld in \cite{FL},
$W^h_m(C)$ is irreducible and generically smooth of dimension $\rho(g,h,m)$
for $[C] \in {\mathcal M}_g$ general. Moreover, for $[L] \in W^h_m(C)$ a smooth point,
$h^0(C,L) = h+1$. Therefore, there is
a unique, reduced component ${\mathcal W}$ of ${\mathcal W}^h_m $ which dominates ${\mathcal M}^0_{g,1}$ via the
map $p_1^{(m)}$.

\noindent
(b) If $\rho(g,h,m) =0$, from \cite[Theorem 1]{EH}, again there is a unique, reduced component ${\mathcal W}$ of
${\mathcal W}^h_m $ dominating ${\mathcal M}^0_{g,1}$. For $[C,x] \in {\mathcal M}^0_{g,1}$ general,
the fibre of ${\mathcal W} \longrightarrow {\mathcal M}^0_{g,1}$ is $W^h_m(C)$ which is a finite number of
points $[L]$ with $h^0(C,L) = h+1$.

Let $${\mathcal Y}_{h,m} := {\mathcal W}
\times_{{\mathcal M}_{g,1}^0} {\mathcal Pic}^{(d-m)},$$which is irreducible.
For $y \in {\mathcal Y}_{h,m}$ general, then $y = (L,N) \in W^h_m(C) \times Pic^{d-m}(C)$ for
$[(C,x)] \in {\mathcal M}^0_{g,1}$ general.

Note also the existence of Poincar\`e line bundles ${\mathcal L}$ and ${\mathcal N}$ on
${\mathcal Y}_{h,m} \times_{{\mathcal M}_{g,1}^0} {\mathcal C}_1 \stackrel{\pi_{h,m}}{\longrightarrow}
{\mathcal Y}_{h,m}$: for $y \in {\mathcal Y}_{h,m}$ general, corresponding to
$(L,N) \in W^h_m(C) \times Pic^{d-m}(C)$ with $[(C,x)] \in {\mathcal M}^0_{g,1}$ general, the restriction of
$\mathcal L$ (respectively, $\mathcal N$) to $\pi_{h,m}^{-1} (y) \cong C$ is $L$ (respectively, $N$).
Hence, we can consider ${\mathcal R}_{h,m} := R^1 (\pi_{h,m})_*(\mathcal N \otimes \mathcal L^{\vee})$ on
${\mathcal Y}_{h,m}$.

There is a dense, open subset $\mathcal U_{h,m} \subset {\mathcal Y}_{h,m}$ on which
the sheaf ${\mathcal R}_{h,m}$ is locally-free and therefore it gives rise to
a vector bundle $ \mathcal E_{h,m}$ whose rank we denote by
$t$: for a general point $y \in \mathcal U_{h,m}$, corresponding to
$(L,N) \in W^h_m(C) \times Pic^{d-m}(C)$, with $[(C,x)] \in {\mathcal M}^0_{g,1}$ general, the fibre of
${\mathcal E}_{h,m}$ on $y$ is $H^1(C, N \otimes L^{\vee}) \cong {\rm Ext}^1(L,N)$.

Note that, if $t = 0$ (this is certainly the case if $d \geq 6g -5$,
cf. Proposition \ref{prop:cases}),
then $\mathcal E_{h,m} \cong \mathcal U_{h,m}$.
On the other hand, if $t>0$, we take into account
the {\em weak isomorphism classes} of extensions (cf. \cite[p. 31]{Fr}). Therefore, we consider
$$\mathcal S_{h,m} := \left\{\begin{array}{cl}
                       \mathcal U_{h,m} & \mbox{if} \; t=0 \\
                       \Pp({\mathcal E}_{h,m}) & \mbox{otherwise}.
                       \end{array}
                       \right.$$

On $\mathcal U_{h,m}$ there is a universal family $\mathcal C_{h,m}$ of curves and on $ \mathcal S_{h,m} \times_{\mathcal U_{h,m}} \mathcal C_{h,m}$
there is a universal vector-bundle $\mathcal F_{h,m}$. A general point $z \in \mathcal S_{h,m}$
corresponds to a pair $(L,N) \in W^h_m(C) \times Pic^{d-m}(C)$, with $[(C,x)] \in {\mathcal M}^0_{g,1}$ general, together with an element $\xi \in \Pp({\rm Ext}^1(L,N))$ if $t>0$; the fibre of $\mathcal F_{h,m}$ on $z$
is the extension $\Ff_z$ of $L$ with $N$ on $C$ corresponding to $\xi$. Given the projection
$\pi_1$ of $ \mathcal S_{h,m} \times_{\mathcal U_{h,m}} \mathcal C_{h,m}$
to the first factor, the sheaf $(\pi_1)_*(\mathcal F_{h,m})$ is free of rank $R+1 =
d - 2g + 2 + h^1 $ on a suitable dense, open subset of $ \mathcal S_{h,m}$. We will  abuse notation
and denote this open subset with $ \mathcal S_{h,m}$. Therefore, on $ \mathcal S_{h,m}$, we have
functions $s_0, \ldots, s_R$ such that, for each point $z \in \mathcal S_{h,m}$, $s_0, \ldots, s_R$
computed at $z$ span the space of sections of the correspoding vector bundle $\Ff_z$.

There is a natural map $\psi_{h,m}: \mathcal S_{h,m} \times {\rm PGL}(R+1, \C) \to
{\rm Hilb}(d,g,h^1)$:
given a pair $(z,\omega)$, embed $\Pp(\Ff_z)$ to $\Pp^R$ via the sections $s_0, \ldots, s_R$
computed at $z$, compose with $\omega$ and take the image.

We define $\HH_{d,g,h^1}^m$ to be the closure of the image of the above map to the
Hilbert scheme. By construction,
$\HH_{d,g,h^1}^m$ dominates ${\mathcal M}_g$ and its general point represents a smooth, linearly
normal scroll $S$ in $\Pp^{R}$ of degree $d$, genus $g$, speciality $h^1$ and containing a unique section
of degree $m$, speciality $h^1$, which is linearly normally embedded in $S$. The general point
of $\HH_{d,g,h^1}^m$ corresponds to an indecomposable scroll if $t>0$; however, in this case,
decomposable scrolls fill up a proper subscheme of $\HH_{d,g,h^1}^m$.

In the next steps, we will show that:
\begin{itemize}
\item  $\HH_{d,g,1}^{2g-2}$ strictly
contains any  $\HH_{d,g,1}^m$, for $m < 2g-2$, and it is an irreducible
component of ${\rm Hilb}(d,g,1)$.
\item for $h^1 \geq 2$, $\HH_{d,g,h^1}^m$ is an irreducible component of ${\rm Hilb}(d,g,h^1)$,
for any $ g + 3 - h^1 \leq m \leq \overline{m}$,
\end{itemize}

\vskip 5pt

\noindent {\bf Step 2. A Lemma concerning automorphisms.}
Here we prove the following:

\begin{lemma}\label{lem:step2}
Assume ${\rm Aut}(C) = \{Id\}$ (in particular, this happens if
$C$ has general moduli). Let \eqref{eq:Fund}
be the exact sequence corresponding to the pair $(S,\Gamma)$, where $S$ is general and
$\Gamma$ is the unique special section of $S$.

If $G_S \subset {\rm PGL} (R +1, \C)$ denotes the sub-group of projectivities of $\Pp^{R}$
fixing $S$,  then $G_S \cong {\rm Aut}(S)$ and
\begin{equation}\label{eq:dimGS}
\dim(G_S)  = \left\{\begin{array}{cl}
                             h^0(N\otimes L^{\vee}) & \mbox{if $\Ff$ is indecomposable} \\
                             h^0(N\otimes L^{\vee}) +1 & \mbox{if $\Ff$ is decomposable} \\
                                                                 \end{array}
                                                                 \right.
\end{equation}
\end{lemma}


\begin{proof}[Proof of Lemma \ref{lem:step2}] 
There is an obvious inclusion $G_S \hookrightarrow
{\rm Aut}(S)$. We want to show that this is an isomorphism. Let $\sigma $ be an automorphism of $S$.
By Theorem \ref{thm:seg1e2}, $\sigma (\Gamma) = \Gamma$ and since
${\rm Aut}(C) = \{Id\}$, $\sigma$ fixes $\Gamma$ pointwise.
Now $H \sim \Gamma + \rho^*(N)$ and by the above,
$\sigma^*(H) = \sigma^*(\Gamma) + \sigma^*(\rho^*(N)) = \Gamma + \rho^*(N) \sim H$. Therefore,
$\sigma$ is induced by a projective trasformation.

The rest of the claim directly follows from cases $(2)$ and $(3)$ of \cite[Theorem 2]{Ma2} and
from \cite[Lemma 6]{Ma2}. Indeed, since ${\rm Aut}(C) = \{Id\}$ therefore, in the
notation of \cite[Lemma 6]{Ma2} one has ${\rm Aut}(S) \cong {\rm Aut}_C(S)$. Furthermore,
from Theorem \ref{thm:seg1e2}, $\Gamma$ is the unique section of minimal degree on $S$. Thus, one can conclude
by using the description of ${\rm Aut}_C(S)$ in \cite[Theorem 2]{Ma2}.
\end{proof}


\noindent {\bf Step 3. The dimension of $\HH^m_{d,g,h^1}$.} Given a
general point of $\HH^m_{d,g,h^1}$ corresponding to a scroll $S$, the base of the scroll $C$
and the line bundles $L$ and $N$ on $C$ are uniquely determined. From the previous steps,
$\dim{\psi_{h,m}}^{-1}([S]) = \dim (G_S)$. An easy computation shows that, in any case,  one has
\begin{equation}\label{eq:nubig}
\dim (\HH_{d,g,h^1}^m) = 5g-5 + (R+1)^2 - h^0(L)h^1(L) - \chi(N \otimes L^{\vee}).
\end{equation}

\vskip 5pt

\noindent {\bf Step 4. The case $h^1=1$.} Let $S$ correspond to a general
point in $\HH^m_{d,g,1}$, i.e. $S$ is determined by a pair $(\Ff,C)$, with $\Ff$ fitting in an exact sequence
like \eqref{eq:Fund}, with $|L|$ a linear series $\g^h_m$ of speciality $1$. Suppose that
$L \neq \omega_C$. In particular, $g \geq 4$.  The residual series is
$ |\omega_C \otimes L^{\vee}|$ is a $\g^0_{2g-2-m}$, i.e.
$\omega_C \otimes L^{\vee} \cong \Oc_C (p_1 + \cdots + p_{2g-2-m})$, where $p_j$ are general
points on $C$, $1 \leq j \leq 2g-2-m$, hence
$L \cong \omega_C(- p_1 - \cdots - p_{2g-2-m})$. From Corollary
\ref{cor:ciro53b} we have that $\HH^m_{d,g,1}$, with $m < 2g-2$, sits in the closure of
$\HH^{2g-2}_{d,g,h^1}$. The dimension count for $\HH^m_{d,g,1}$'s follows
by \eqref{eq:nubig}. The generic smoothness of $\HH^{2g-2}_{d,g,1}$
follows by comparing the formula
for $h^0(\N_{S/\Pp^R})$ from Proposition \ref{prop:tghilbh1} (equivalently \eqref{eq:contoesplicito})
with \eqref{eq:nubig}.

The dimension count  for $\HH^m_{d,g,1}$ follows by the construction.

\vskip 5pt

\noindent {\bf Step 5. The case $h^1 \geq 2$.} In Step 1 we constructed the irreducible
subschemes $\HH_{d,g,h^1}^m$ of the Hilbert scheme ${\rm Hilb}(d,g,h^1)$ for  $g+3-h^1 \leq m \leq \overline{m}$.
From the dimension count \eqref{eq:nubig} and from the formula for $h^0(\N_{S/\Pp^R})$ in Proposition \ref{prop:tghilbh1} (equivalently \eqref{eq:contoesplicito})
one has that each $\HH_{d,g,h^1}^m$ is a generically smooth, irreducible component of the Hilbert scheme
${\rm Hilb}(d,g,h^1)$.

\vskip 5pt

We have finished the proof of Theorem \ref{thm:hilschh1}.
\end{proof}

\begin{remark}\label{rem:maruyama} Observe that \eqref{eq:dimGS} can be also
computed via cohomological arguments.
Indeed, $\dim({\rm Aut}(S)) = h^0(S,\T_S)$. From \eqref{eq:tgrel}, one has $H^0(\T_S) \cong H^0(\T_{rel})$.
Since $\T_{rel} \cong \omega^{\vee}_S \otimes \rho^*(\omega_C)$, we get $\T_{rel} \cong \Oc_S (2H - \rho^*({\rm det}(\Ff)))$.
As $\Gamma \cdot (2H-\rho^*({\rm det}(\Ff))) = 2m-d  = \Gamma^2 <0$,
$\Gamma$ is a fixed component
of $|2H- \rho^*({\rm det}(\Ff))|$.
Since $H \sim \Gamma + \rho^*(N)$, then $h^0(\Oc_S(2H-\rho^*({\rm det}(\Ff))))  = h^0(\Oc_S(H- \rho^*(L)))$. Moreover,
$\Gamma \cdot (H-\rho^*(L)) =0$.
By the projection formula, we get $H^0(S, H- \rho^*(L)) \cong H^0(C, \Ff \otimes L^{\vee})$. Therefore,
$h^0(S,\T_S) = h^0(C, \Ff \otimes L^{\vee})$.

The map $H^0(\Ff \otimes L^{\vee}) \to  H^0(\Oc_C)$ arising from
the exact sequence \eqref{eq:Fund} identifies with the restriction map
$H^0(\Oc_S( H-\rho^*(L)) ) \to H^0(\Oc_{\Gamma})$.

\noindent
$\bullet$ If $\Ff$ is decomposable, this map is surjective. Therefore,
$h^0(\Ff\otimes L^{\vee}) = 1 + h^0( N \otimes L^{\vee})$. In particular, this happens
when ${\rm Ext}^1(L,N) \cong H^1(N \otimes L^\vee) = 0$, e.g. if $d \geq 6g - 5$ (cf. Proposition \ref{prop:cases}).

\noindent
$\bullet$ If $\Ff$ is indecomposable, then the coboundary map $H^0(\Oc_C) \to H^1(N \otimes L^\vee)$, arising
from \eqref{eq:Fund}, is injective since it corresponds to the choice of
$\Ff$ as an element of $ {\rm Ext}^1(L,N) $. In particular, $h^0(\Ff\otimes L^{\vee}) = h^0( N \otimes L^{\vee})$.

The above discussion proves \eqref{eq:dimGS}.
\end{remark}

\begin{remark}\label{rem:GP} As an alternative to the construction of  $\HH_{d,g,h^1}^{m}$ in the above proof of
Theorem \ref{thm:hilschh1},
one could start from the main result in \cite{GP3}. One has that the general scroll in $\HH_{d,g,1}^{2g-2}$
can be obtained as a general internal projection of a scroll $S$
corresponding to a decomposable vector-bundle on a genus $g$
curve $C$ of the type $\omega_C \oplus M$, with $M$ a non-special,
line-bundle of degree greater or equal than $2g-2$.
Such a scroll has a unique, special section which is canonical.

One may prove that  the general scroll in
$\HH_{d,g,h^1}^{m}$ arises as the projection of $S$ from suitable $2g-2-m$ points on the canonical
section and from other general points on $S$.
\end{remark}

\section{Further considerations on Hilbert schemes}\label{S:considerations}

\subsection{Components of the Hilbert scheme of non-linearly normal scrolls}

In this section we consider the following problem.
Let $r = d - 2g + 1 + k$, with $0 \leq k < l$ and
consider the family $\mathcal Y^m_{k,l}$ whose general element is a
general projection to $\Pp^r$ of the general scroll in $ \HH_{d,g,l}^m$. Is  $\mathcal Y^m_{k,l}$
contained in $\HH_{d,g,k}^n$ for some $n$, if $k>0$, or in  $\HH_{d,g}$, if $k=0$? Recall that $\HH_{d,g}$ is
the component of the Hilbert scheme of linearly normal,
non-special scrolls of degree $d$ and genus $g$ in $\Pp^{d-2g+1}$ (cf. \cite{CCFMLincei}).

\begin{proposition}\label{prop:ciropasqua} In the above setting:
\begin{itemize}
\item[(i)] if $k>0$, $\mathcal Y^m_{k,l}$ sits
in an irreducible component of the Hilbert scheme different from $\HH_{d,g,k}^n$, for any $n$;
\item[(ii)] if $l>1$, $\mathcal Y^m_{0,l}$ sits in an irreducible
component of the Hilbert scheme different from $\HH_{d,g}$, for any $m$;
\item[(iii)] $\mathcal Y^{2g-2}_{0,1}$ is a divisor inside $\HH_{d,g}$ whose general point is a smooth
point for $\HH_{d,g}$.
\end{itemize}
\end{proposition}

\begin{proof} (i) It suffices to prove that
\begin{equation}\label{eq:ciropasqua}
\dim(\mathcal Y^m_{k,l}) \geq \dim (\HH_{d,g,k}^m).
\end{equation}Indeed, if
$\mathcal Y^m_{k,l}$ is contained in a component $\HH_{d,g,k}^n$, then $n \geq m$ and
the conclusion follows from Remark \ref{rem:hilschh1}.

In order to prove \eqref{eq:ciropasqua}, we count the number
of parameters on which $\mathcal Y^m_{k,l}$ depends.
Let $[S] \in \HH_{d,g,l}^m$ be general. Then, $S \subset \Pp^R$, where $R=
d-2g+1+l$ and $S \cong \Pp(\Ff)$ with
$\Ff$ a vector bundle on a curve $C$ sitting in an exact sequence like \eqref{eq:Fund}.
Let $S' \subset \Pp^r$ be the general projection of $S$, with $r$ as above.
Let $G_{S'} \subset {\rm PGL}(r+1, \C)$ be the subgroup of projectivities which fix $S'$.

The parameters on which $\mathcal Y^m_{k,l}$ depends are the following:
\begin{itemize}
\item $3g -3$, for the parameters on which $C$ depends, plus
\item $g$, for the parameters on which $N$ depends, plus
\item $\rho(g,L) = g - l (m-g+l+1)$, for the parameters on which $L$ depends, plus
\item $\epsilon = 0$ (respectively, $h^1(N\otimes L^{\vee}) - 1$) if the general bundle is
decomposable (respectively, indecomposable), plus
\item $(r+1) (l-k) = dim (\G (r, R)) $, which are the parameters for the projections, plus
\item $(r +1)^2 -1 = dim ({\rm PGL}(r+1, \C))$, minus
\item $\dim(G_{S'})$.
\end{itemize} We remark that, since $G_S \cong {\rm Aut}(S)$ (cf. Lemma \ref{lem:step2}), one has $\dim(G_{S'}) \leq \dim(G_S)$.

Therefore, by recalling \eqref{eq:dimGS}, in any case we get
\begin{equation}\label{eq:dimY}
\dim(\mathcal Y^m_{k,l}) \geq 5g-5 + (r+1)^2 - \chi(N\otimes L^{\vee}) - l(m-g+l+1) + (l-k) (r+1).
\end{equation}

By Theorem \ref{thm:hilschh1} and \eqref{eq:contoesplicito},
$$\dim(\HH_{d,g,k}^m) = 5g-5  + (r+1)^2 - k (m-g+k+1) - \chi(N \otimes L^{\vee}).$$Then
$$\dim({\mathcal Y}_k) - \dim(\HH_{d,g,k}^m) \geq (l-k) (d-m - g + 1 - k -l).$$From the assumptions on  $d$ and
the facts that  $L$ is special and $k <l \leq \frac{g}{4}$, then $\dim(\mathcal Y^m_{k,l}) > \dim(\HH_{d,g,k}^m)$, i.e.
we proved \eqref{eq:ciropasqua}.

\noindent
(ii) In this case we have to prove
$$\dim(\mathcal Y^m_{0,l}) \geq \dim (\HH_{d,g}) = (r+1)^2 + 7(g-1).$$Arguing as above, we see that
this is a consequence of
$$l (d-g-l+1 -m) \geq g-1+d-2m,$$which holds since
$l \geq 2$ and by the assumptions on $d$.

\noindent
(iii) The same computation as above shows that
$$\dim (\mathcal Y^{2g-2}_{0,1}) \geq \dim (\HH_{d,g}) -1.$$We want to prove that equality holds and that
$\mathcal Y^{2g-2}_{0,1} \subset \HH_{d,g}$. Consider the {\em Rohn exact sequence}
$$0 \to \Oc_S(H) \to \N_{S/\Pp^{r+1}} \to \N_{S'/\Pp^{r}} \to 0$$(see, e.g. \cite{Cil}, p. 358, formula (2.2)).
From Proposition \ref{prop:tghilbh1} (ii), we have $h^1(\N_{S/\Pp^{r+1}}) =0$, therefore
also $h^1(\N_{S'/\Pp^{r}}) = 0$. Hence $\mathcal Y^{2g-2}_{0,1}$ is contained in a component $\HH$
of the Hilbert scheme of dimension $\chi(\N_{S'/\Pp^{r}} ) = 7(g-1) + (r+1)^2$ and
the general point of $\mathcal Y^{2g-2}_{0,1}$ is a smooth point of $\HH$. The map
$H^0(\N_{S/\Pp^{r+1}}) \to H^0(\N_{S'/\Pp^{r}}) $ is not surjective:
its cokernel is $H^1(\Oc_S(H)) $, which has dimension one (cf. Proposition \ref{prop:tghilbh1} (ii)). This shows
that the general point of $\HH$ correspond to a smooth scroll which is linearly normal in $\Pp^r$.
By the results in \cite{CCFMLincei}, $\HH = \HH_{d,g}$.
\end{proof}

\begin{remark}\label{rem:Lincei} The previous result extends and makes more precise the contents of
\cite[Example 5.12]{CCFMLincei}.
\end{remark}

\subsection{Components of the Hilbert scheme with special moduli}

There are irreducible components of ${\rm Hilb}(d,g,h^1)$
which do not dominate $\M_g$, i.e. components with {\em special moduli}. We prove the existence of some
of these components in the next example.
To do this we first recall some preliminary results we will use.

Given any integer $g \geq 3$, let
\begin{equation}\label{eq:gon}
\gamma := \left\{ \begin{array}{cl}
             \frac{g+2}{2} & {\rm if} \;\; g \;\; {\rm even},\\
 \frac{g+3}{2} & {\rm if} \;\; g \;\; {\rm odd}.
           \end{array}
  \right.
\end{equation}Set  \[\M^1_{g,t}:= \{[C] \in \M_g
| \; C \; \mbox{possesses a  } \;
\mathfrak{g}^1_t \},
\]which is called the $t$-{\em gonal locus} in $\M_g$.

It is well-known that $\M^1_{g,t}$ is irreducible, of
dimension $2g + 2 t -5$, when $ t < \gamma$, whereas
$\M^1_{g,t}= \M_g$, for $t \geq \gamma$
(see e.g. \cite {AC2}). Moreover, the general curve in $\M^1_{g,t}$ has no non-trivial automorphism
(cf. the computations as in \cite[p. 276]{GH}). Finally, for $t<\gamma$,
the general curve in $\M^1_{g,t}$ has a unique base-point-free $\mathfrak{g}^1_t $ (\cite[Theorem 2.6]{AC2}).

\begin{proposition}\label{prop:ballico} (cf. \cite[Prop. 1]{Ballico}) Fix positive integers $g$, $t$, $r$ and $a$, with
$a \geq 3$, $(a-2)(t-1) < g \leq (a-1) (t-1)$. Let $|D|$ be the b.p.f. linear series
$\g^1_t$ on a general $t$-gonal curve $C$ of genus $g$. Then $$\dim (|rD|) = r, \;\; {\rm if} \; r \leq a-2.$$
\end{proposition}

With  assumptions as in Proposition \ref{prop:ballico}, if we consider
\begin{equation}\label{eq:lr}
L_r := \Oc_C(K_C - r D)
\end{equation}then $L_r$ is a special line bundle, of speciality $r+1$; in particular,
$h^0(L_r) = g - r(t-1)$.

\begin{example}\label{ex:specmod}
Let $a$, $g$ and $t$ be positive integers as in Proposition \ref{prop:ballico}.
Let $[C] \in \M^1_{g,t}$ be general, with
$2 < t < \gamma$. Denote by $D$ the divisor on $C$ such that $|D|$ is the $\mathfrak{g}^1_t$ on $C$.
Let $ 2 \leq l \leq a-1$ be any positive integer. As in \eqref{eq:lr},
let $$L:= L_{l-1} = \omega_C \otimes \Oc_C (-(l-1)D).$$Then
$$m:= {\rm deg}(L) = 2g-2-(l-1)\;t \;\; {\rm and} \;\; h^1(L) = l.$$If we further assume that
\begin{equation}\label{eq:KK}
l \leq \frac{2g}{t(t-1)} - \frac{1}{t} -1,
\end{equation}then $L$ is also very ample (cf. \cite[Theorem B]{KK}).

Let $d \geq 6g-5$ be an integer.
Let $N \in Pic^{d-m}(C)$ be general. Then ${\rm deg}(N) = d+ (l-1) t -2g+2 \geq 4g-3+ (l-1) t$.
Therefore, $N$ is very ample and non-special.

Let $\Ff= N \oplus L$: the choice of $\Ff$ decomposable is not restrictive;
indeed, since $d \geq 6g-5$ and $L$ is special, from Proposition \ref{prop:cases} any scroll in ${\rm Hilb}(d,g,l)$
is associated with a splitting vector bundle. From Theorem \ref{thm:seg1e2}, the pair $(\Ff,C)$
determines a smooth, linearly normal scroll $S \subset \Pp^{R}$, $R = d-2g+1 + l$,
of degree $d$, genus $g$, speciality $l$,
with special moduli and containing a unique special section $\Gamma$, which corresponds to $L$ (cf. Theorem
\ref{thm:seg1e2} (i)).

Scrolls arising from this constructions fill-up  closed subschemes ${\mathcal Z}_{t,l}$
of ${\rm Hilb}(d,g,l)$, which depend on the following parameters:
\begin{itemize}
\item $2g +2t-5$, since $C$ varies in $\M^1_{g,t}$, plus
\item $g$, which are the parameters on which $N$ depends, plus
\item $(R+1)^2 -1 = dim ({\rm PGL}(R+1, \C)),$ minus
\item $\dim(G_S)$, which is the dimension of the projectivities of $\Pp^{R}$ fixing a
general $S$ arising from this construction.
\end{itemize}Since $\Ff$ is decomposable, from \eqref{eq:dimGS} it follows that
\begin{equation}\label{eq:dimZ}
\dim({\mathcal Z}_{t,l}) = (R+1)^2 + 8(g-1) - 4 - d - 2t(l-2).
\end{equation}

On the other hand, if we assume $g \geq 4l$, it makes sense to consider also
the irreducible components  $\HH_{d,g,l}^m$ of ${\rm Hilb}(d,g,l)$ with general moduli,
which have been constructed in Theorem \ref{thm:hilschh1}.

We claim that ${\mathcal Z}_{t,l}$ is not contained in any component of the type $\HH_{d,g,l}^m$, for any
$m$.  From Remark \ref{rem:hilschh1} it suffices to show that $\dim ({\mathcal Z}_{t,l}) \geq
\dim(\HH_{d,g,l}^m)$, with $m = 2g-2 - (l-1)t$. In this case, by Theorem \ref{thm:hilschh1} - (ii), we get
\begin{equation}\label{eq:dimHml}
\dim (\HH_{d,g,l}^m) = (10-l) (g-1) - d - l^2 + t(l-1)(l-2) + (R+1)^2.
\end{equation}Thus, by using \eqref{eq:dimZ} and \eqref{eq:dimHml}, we get
\begin{equation}\label{eq:diff}
\dim({\mathcal Z}_{t,l}) - \dim (\HH_{d,g,l}^m) = (l-2) (g+1+l - t(l+1)).
\end{equation}From \eqref{eq:KK}, one has
$$g+1+l - t(l+1) \geq 2+l + g \; \frac{t-3}{t-1} >0$$since $t \geq 3$. This implies that
the difference is non-negative and therefore the assertion.
\end{example}

In the case $h^1=2$, we can be even more precise.

\begin{proposition}\label{prop:h1=2} The irreducible components of ${\rm Hilb}(d,g,2)$ are either $\HH_{d,g,2}^m$ or
${\mathcal Z}_{t,2}$.
\end{proposition}
\begin{proof} By Theorem \ref{thm:seg1e2}, any scroll in ${\rm Hilb}(d,g,2)$ either belongs to a component
$\HH_{d,g,2}^m$ or to  ${\mathcal Z}_{t,2}$. Since, from Remark \eqref{rem:hilschh1}  and
from \eqref{eq:diff}, all of them have the same dimension (in particular, independent from $t$) this
proves the assertion.
\end{proof}

\begin{remark}\label{rem:19608} There are examples of components of the Hilbert scheme of linearly normal, smooth,
special scrolls in $\Pp^R$, of degree $d$, genus $g$ and speciality $h^1$, with $g < 4h^1$.
According to Corollary \ref{cor:seg1e2}, they have special moduli. Examples of such components are e.g.
the $\mathcal Z_{3,l}$'s, with $l > 4$ and $g = 3l +4 < 4l$. Note that \eqref{eq:KK} is verified in this case
with equality. From Proposition \ref{prop:ballico}, $a = \lceil \frac{g}{2} \rceil +1$ hence
$l \leq a-1$. Therefore, the construction in Example \ref{ex:specmod} works and produces
the examples in questions.

\end{remark}

\subsection{Singular points of the Hilbert scheme}

Finally we will prove the existence of points of components of the type $\HH_{d,g,h^1}^m$ corresponding
to smooth scrolls with general moduli which are singular points of the Hilbert scheme.

\begin{proposition}\label{prop:16apr} Assumptions as in Theorem \ref{thm:hilschh1}.
Let $[S] \in \HH_{d,g,h^1}^m$ be a smooth scroll with the special section $\Gamma$ of degree $m$ corresponding
to $[L] \in W^r_m(C)$ such that $\omega_C \otimes L^{\vee}$ has base points. Then $[S]$ is a singular point of the Hilbert scheme.
\end{proposition}

\begin{proof} It is an immediate consequence of Theorem \ref{thm:hilschh1} and of Remark \ref{rem:rinvio}.
\end{proof}

\begin{remark}\label{rem:16apr} One has examples of such singular points of the Hilbert scheme
as soon as $\rho(g, h+1, m+1) = \rho(g, h^1-1, 2g-3-m) \geq 0$. This is equivalent to
the inequality $g \geq h^1 \frac{m+h^1+2}{h^1+1}$.
\end{remark}

Similarly, the presence of base points for the linear series $|L|$ also produces singular points
of the Hilbert scheme.

\begin{proposition}\label{prop:16aprb} Let $[S] \in \HH_{d,g,h^1}^m$ correspond to
a smooth scroll with general moduli and
with the special section $\Gamma$ of degree $n< m$.
Then $[S]$ is a singular point of the Hilbert scheme.
\end{proposition}

\begin{proof} Such a point $[S]$ also belongs to the component $\HH_{d,g,h^1}^n$.
\end{proof}

\begin{remark}\label{rem:16aprb} The existence of such singular points is ensured by
Corollary \ref{cor:ciro53b}.  This implies that $\HH_{d,g,h^1}$ is connected.
\end{remark}



\end{document}